\documentclass[12pt]{article}
\usepackage{amsmath,amssymb,amsbsy,amsfonts,amsthm,latexsym,
                                    amsopn,amstext,amsxtra,euscript,amscd}

\newtheorem{thm}{Theorem}
\newtheorem{theorem}[thm]{Theorem}

\newtheorem{lemma}[thm]{Lemma}
\newtheorem{prop}{Proposition}
\newtheorem{rem}{Remark}
\newtheorem{proposition}[prop]{Proposition}
\newtheorem{cor}[thm]{Corollary}

\DeclareMathOperator{\lcm}{lcm}

\def\li{\mathrm{li}\,}

\def\mand{\qquad \mbox{and} \qquad}

\def\\{\cr}
\def\({\left(}
\def\){\right)}
\def\[{\left[}
\def\]{\right]}
\def\<{\langle}
\def\>{\rangle}
\def\fl#1{\left\lfloor#1\right\rfloor}

\def\cA{\mathcal A}

\def\cC{\mathcal C}

\def\cI{\mathcal I}

\def\cQ{\mathcal Q}
\def\cR{\mathcal R}
\def\cS{\mathcal S}
\def\cU{\mathcal U}

\def\eps{\varepsilon}

\def\E{\mathbf{E}}
\def\F{\mathbb{F}}

\def\Z{\mathbb{Z}}
\def\K{{K}}
\def\L{{L}}
\def\Q{\mathbb{Q}}

\def\cUp{\cU_p}
\def\cUpk{\cU_{p,k}}
\def\cRpk{\cR_{p,k}}

\begin{document}

\title{Rank statistics for a family of elliptic curves over
a function field}

\author{
{\sc Carl Pomerance}\\
{Department of Mathematics, Dartmouth College}\\
{Hanover, NH 03755-3551, USA} \\
{\tt carl.pomerance@dartmouth.edu} \\
\and
{\sc Igor E. Shparlinski}\\
             {Department of Computing, Macquarie University}\\
{ Sydney, NSW 2109,
Australia}\\
{\tt igor@ics.mq.edu.au}
}

\pagenumbering{arabic}

\date{ }
\maketitle

\bigskip
\bigskip
\hfill{\em{Dedicated to John Tate}}
\bigskip
\bigskip

\begin{abstract} We show that the average
and typical ranks in a
certain parametric family of elliptic curves
described by D.~Ulmer tend to infinity  as
the parameter $d \to\infty$. This is perhaps
unexpected since by a result of A.~Brumer,
the average rank for all elliptic curves over
a function field of positive characteristic is
asymptotically bounded above by 2.3.
\end{abstract}

{\renewcommand{\thefootnote}{}
\vskip -25pt
\footnote{
\vskip 0pt
\noindent The first author was supported in part by NSF grant DMS-0401422.
The second author was supported in part by ARC grant DP0556431.}
}

\newpage

\section{Introduction}

\subsection{Background}

Let $\F_q$ be the finite field of $q$ elements
of prime characteristic $p$.
We consider the parametric family of curves
$$
\E_d: \quad y^2 + xy = x^3 - t^d
$$
over the function field $\F_q(t)$, where $d$ is a positive integer.
Among other results, Ulmer~\cite[Proposition~6.4]{Ulm}
has shown that the conjecture of Birch and Swinnerton-Dyer holds for
each $\E_d$ when $d$ is not divisible $p$.

Denote by  $\cUp$  the set of positive integers which divide
some member of the sequence $p^n +1$, for $n=1,2,\dots$.
Let $\varphi$ denote Euler's function, and for $a,b$
coprime integers with $b>0$, let $\ell_a(b)$ be
the multiplicative order of the residue class $a$ in
the group $(\Z/b\Z)^\times$.   We always have
$\ell_a(b)\mid \varphi(b)$.
Ulmer~\cite[Theorem~9.2]{Ulm} has also shown that for every $d \in \cUp$,
the rank $R_q(d)$ of $\E_d$ over $\F_q(t)$ is given by
\begin{equation}
\label{eq:R and I}
R_q(d)~ =~ I_q(d)-C_q(d),
\end{equation}
where
$$
I_q(d) ~= ~\sum_{e|d} \frac{\varphi(e)}{\ell_q(e)}
$$
and $C_q(d)$ is an explicit correction term that
always satisfies $0\le C_q(d) \le 4$.
(Note that $d\in\cUp$ implies that $\gcd(e,q)=1$ for each
$e\mid d$, so that $I_q(d)$ is defined.)
Since members of $\cU_p$ are coprime to $p$,
the Birch and Swinnerton-Dyer conjecture holds for $\E_d$ for $d\in\cU_p$,
so that~\eqref{eq:R and I} holds as well for the analytic rank.

Ulmer~\cite{Ulm} considers the specific case $d=p^n+1$ and
$q=p$.  Then $\ell_p(d)=2n$, and each $\ell_p(e)\mid 2n$, so that
$$
I_p(p^n+1)~\ge~\sum_{e\mid p^n+1}\frac{\varphi(e)}{2n}~= ~\frac{p^n+1}{2n}.
$$
Thus,
$$
R_p(d)~\ge~ \frac{d \log p}{2\log d}-4,
$$
which compares very nicely with the upper bound
$$
R_p(d) ~\le ~\frac{d \log p}{2\log d} +O\(\frac{d (\log p)^2}{(\log d)^2}\)
$$
(uniformly over $d$ and $p$) due to Brumer~\cite{Bru}.

It is interesting that the expression $I_q(d)$ occurs in other contexts.
For example, Moree and Sol\'e~\cite{Mor} show that $I_q(d)$ is the
number of irreducible
factors of $t^d-1$ in $\F_q[t]$ and go on to apply $I_q(d)$ to a
combinatorial problem.

\subsection{Our results}

Using~\eqref{eq:R and I}, we show that on average over all numbers $d$
(without the restriction that $d \in \cUp$), the rank of $\E_d$
is quite large.  We do not know how to bound the
rank from above for integers
$d\not\in\cU_p$, but we can show that the average over $\cU_p$ is
not quite as big as Brumer's upper bound.

\begin{theorem}
\label{thm:Average}
There exists an absolute constant $\alpha > 1/2$ such that
for all finite fields $\F_q$ and all sufficiently large
large values of $x$ (depending only on the characteristic $p$ of $\F_q$),
\begin{equation}
\label{eq:alpha}
\frac{1}{x} \sum_{d\le x} R_q(d) ~ \ge~ x^\alpha.
\end{equation}
Moreover, for $x$ sufficiently large depending on $q$,
\begin{equation}
\label{eq:upper}
\Bigg(\sum_{\substack{d\le x\\ d\in\cU_p}}1\Bigg)^{-1}
\sum_{\substack{d\le x\\ d\in\cU_p}}R_q(d)~\le~
x^{1-\log\log\log x/(2\log\log x)}.
\end{equation}
\end{theorem}

The constant $\alpha$ in~\eqref{eq:alpha}
can be explicitly evaluated.
Moreover, assuming the Elliott--Halberstam conjecture about the
distribution of primes in residue classes (described below), we can show that
$\alpha$ may be taken as any number smaller
than~1.  Probably the upper bound~\eqref{eq:upper} is close to the
truth, but we do conjecture that the ``2" in the denominator of
the exponent can be removed.

The average order is presumably skewed by a few numbers $d$ where
the rank is especially big, at least that is the way we prove the
lower bound in
Theorem~\ref{thm:Average}.  One might wonder about $R_q(d)$ for a
``typical" number $d$.
We show that for almost all numbers $d$, in the sense of asymptotic density,
the rank is still fairly large.

\begin{theorem}
\label{thm:Normal}
Let $\F_q$ be a finite field of characteristic $p$
and let $\eps>0$ be arbitrary.   As $x\to\infty$,
except for $o_{p,\eps}(x)$ values of $d\le x$, we have
$$
R_q(d) ~\ge ~(\log d)^{(1/3-\eps)\log\log\log d}.
$$
\end{theorem}

It has been shown by Brumer~\cite{Bru} that the average analytic rank
over all elliptic curves over a function field of positive
characteristic is bounded above by 2.3 asymptotically.  Since
by a result of Tate~\cite{Tate}
the algebraic rank is bounded by the analytic rank, the
same bound holds as well for the algebraic rank.  Thus,
Theorems~\ref{thm:Average} and~\ref{thm:Normal} show
that the thin family consisting of the curves $\E_d$
is indeed very special.

Concerning the set $\cU_p$ for which the rank formula~\eqref{eq:R and I}
holds, we show that the number of elements in $\cU_p$ up to $x$
is asymptotic to $c_px/(\log x)^{2/3}$ as $x\to\infty$, where $c_p$
is a positive constant,
see Corollary~\ref{cor:Dpdistribution} below.
(A more precise formula
may be found in Moree~\cite[Theorem 5]{Mo}.)

We remark that it seems very plausible that using the methods
of~\cite{EPS} and~\cite{LiPom} one can show
that under the assumption of the
Generalized Riemann Hypothesis for Kummerian fields
over $\Q$, we have
$$
R_q(d)=(\log d)^{(1+o(1))\log\log\log d}
$$
for almost all numbers $d\in\cU_p$ in the sense of asymptotic
density.   We hope to take this
up in a future paper.

Perhaps more importantly, it should be interesting to investigate
the situation for more families of elliptic curves than the one
family of Ulmer that we consider here.  For example, in
Darmon~\cite{Dar} many other families are considered each
of a similar flavor to Ulmer's.  One might
not know the Birch and Swinnerton-Dyer
conjecture in these cases, but at least some statistical
information might be gleaned for the analytic ranks.

\medskip
\noindent{\bf Acknowledgment.}
We wish to thank Douglas Ulmer for some helpful comments and his
encouragement.  We also thank an anonymous referee for a careful reading.

\section{Preparations}

\subsection{Notation}

We always use the letters $l$, $p$, $r$, $s$, and $t$ to denote prime numbers,
while $d$, $e$, $k$, $m$, and $n$ always denote positive integers.
We let $P(n)$ denote the largest prime factor of $n$ if $n>1$, and $P(1)=1$.

As usual,  we use $\pi(x;k,a)$ to denote the number of primes $r\le x$ with
$r\equiv a \pmod k$, and we let $\pi(x)$ denote the total number of
all primes $r\le x$.

Given a set $\cA$ of positive integers, we use $\cA(x)$ to
denote the subset of $a \in \cA$ with $a \le x$.

For any real number $x>0$
and any integer $\nu\ge 1$, we write $\log_\nu x$
for the function defined inductively by
$\log_1 x=\max\{\log x,\,1\}$ (where $\log x$ is the natural
logarithm of $x$) and $\log_\nu x=\log_1(\log_{\nu-1} x)$
for $\nu>1$.

We use the order symbols $O$, $o$, $\ll$, $\gg$ with their
usual meanings in analytic number theory, where  all implied
constants are {\it absolute\/}, unless
indicated by subscripts.  (We recall
that the notations $A\ll B$, $B\gg A$ and $A=O(B)$ are equivalent.)

We use $v_l(n)$ to denote the (exponential) $l$-adic valuation of $n$; that is,
$v_l(n)$ is the exponent on the prime $l$ in the prime
factorization of $n$.

\subsection{Structure of $\cUp$}
\label{sec:Struct D}

Recall that $\cUp$ is the set of natural numbers that divide $p^n+1$
for some positive integer $n$.

\begin{lemma}
\label{lem:structure}
Let $p$ be a prime number and suppose $d\in\cUp$.
\begin{itemize}
\item[(i)]
There is a positive integer $k$ such that
$v_2(\ell_p(r))=k$ for each odd prime factor $r$ of $d$.
\item[(ii)]
If $p>2$ and $k=1$, then $v_2(d)\le v_2(p+1)$, while if
$p>2$ and $k>1$, then $v_2(d)\le 1$.
\end{itemize}
\end{lemma}
\begin{proof}
Suppose $d\in\cUp$ and $r$ is an odd prime factor of $d$.
Since $d\mid p^n+1$ for some positive integer $n$,
we have $r\mid p^n+1$ and $r\nmid p^n-1$.
Thus, $\ell_p(r)\mid 2n$ and
$\ell_p(r)\nmid n$, so $v_2(\ell_p(r))=v_2(2n)=v_2(n)+1$.
Thus, (i) follows with $k=v_2(n)+1$.
For (ii) note that from our proof of (i), $k=1$ if and only if
$n$ is odd.  But for odd $n$ we have $v_2(p^n+1)=v_2(p+1)$, so
$v_2(d)\le v_2(p+1)$.  And if $k>1$, we have $n$ even, so
$p^n+1\equiv 2\pmod 4$ and $v_2(d)\le 1$.
\end{proof}

For $p$ prime and $k$ a positive integer let
$\cUpk$ denote the set of integers $d$ coprime to $p$
such that for each odd prime $r\mid d$ we have $v_2(\ell_p(r))=k$;
further, if $p>2$, $k=1$, then $v_2(d)\le v_2(p+1)$, and if $p>2$, $k>1$,
then $v_2(d)\le 1$.
Thus, Lemma~\ref{lem:structure} implies that
$\cUp\subset\bigcup_{k\ge1}\cUpk$.  In fact, they are equal.

\begin{lemma}
\label{lem:structure2}
For each prime $p$, we have $\cUp=\bigcup_{k\ge1}\cUpk$.
\end{lemma}
\begin{proof}
Suppose $d\in\cUpk$.  We may assume $d>2$.  If $d$ is a power of 2, then $k=1$,
$p>2$, and $d\mid p+1$, so that $d\in\cUp$.  If $d$ is not a power of 2, let
$d_o$ be the odd part of $d$ and let
$m=\ell_p(d_o)$.  Then $m$ is the least common multiple  of the numbers
$\ell_p(r^a)$
where $r^a$ runs over the odd prime power divisors of $d$.  We
have $\ell_p(r^a)/\ell_p(r)\mid r^{a-1}$, so that if $r$ is
odd, we have $v_2(\ell_p(r^a))=v_2(\ell_p(r))=k$.  Thus, $v_2(m)=k$ and
we have $r\nmid p^{m/2}-1$.  But $r^a\mid p^m-1$, so we have
$r^a\mid p^{m/2}+1$.  Thus, the odd part of $d$ divides $p^{m/2}+1$.
If $k>1$ and $p>2$, then $v_2(d)\le1$, so that the even part of $d$
also divides $p^{m/2}+1$.  Further, if $k=1$ and $p>2$, then
$v_2(d)\le v_2(p+1)$.
In this case, $m/2$ is odd, so that $p+1\mid p^{m/2}+1$,
and so the even part of $d$ again divides $p^{m/2}+1$.  We thus have
that $d\mid p^{m/2}+1$, and this concludes the proof.
\end{proof}

Let $\cRpk$ denote the set of odd prime members of $\cUpk$.  That is,
$$
\cRpk~=~\{r~\hbox{an odd prime}~: r\ne p,~v_2(\ell_p(r))=k\}.
$$
Then,
$\cUpk$ is the set of integers $d$ all of whose odd prime factors come from
$\cRpk$, with $v_2(d)$ bounded as discussed above. After a classical
result of Wirsing~\cite{Wirs}, the distribution of the sets $\cUpk$
within the natural numbers follows from the distribution of the sets
$\cRpk$ within the prime numbers in a way that is made more precise below.

The following result should be compared with results in~\cite{Mo}
and with~\cite[Theorem~1.3]{Pappa}.
We discuss the proof in Section~\ref{sec:Prop1 proof}.

\begin{proposition}
\label{prop:Rpkdistribution}
Let $x$ be large and let $p\le(\log x)^{2/3}$ be a prime number.
Let
$$E(x)=\frac{x\log_2x}{(\log x)^{7/6}}
$$
For $p >2$, we have
\begin{eqnarray*}
\#\cR_{p,1}(x)~=~\frac13\pi(x)+O(E(x)),&&
\#\cR_{p,2}(x)~=~\frac16\pi(x)+O(E(x)),\\
\sum_{k\ge3}\#\cRpk(x)&=&\frac16\pi(x)+O(E(x)).
\end{eqnarray*}
Further,
\begin{eqnarray*}
\#\cR_{2,1}(x)~=~\frac7{24}\pi(x)+O(E(x)),
&&
\#\cR_{2,2}(x)~=~\frac13\pi(x)+O(E(x)),\\
\sum_{k\ge3}\#\cR_{2,k}(x)&=&\frac1{12}\pi(x)+O(E(x)).
\end{eqnarray*}
\end{proposition}

For $p$ a prime, let
$$
\cR_p=\begin{cases}
\cR_{p,1}\,,&p>2\\
\cR_{2,2}\,,&p=2.
\end{cases}
$$
{}From Proposition~\ref{prop:Rpkdistribution} we have
\begin{equation}
\label{eq:Rpdistribution}
\#\cR_p(x)~=~\frac13\pi(x)+O\(\frac{x\log_2x}{(\log x)^{7/6}}\).
\end{equation}

We can now establish the following result about the distribution of the sets
$\cUp$.

\begin{cor}
\label{cor:Dpdistribution}
For each prime $p$,
there is a positive constant $c_p$ such that
$$
\#\cUp(x)~\sim~c_px/(\log x)^{2/3}
$$
as $x\to\infty$.
\end{cor}

\begin{proof}
It follows directly from Proposition~\ref{prop:Rpkdistribution}
and
Wirsing's theorem~\cite{Wirs}
(see too~\cite[Chapter~II.7, Exercise~9]{Ten})
that there are positive constants $c_p$ such that
$$
\#\cU_{p,1}(x)~\sim~c_px/(\log x)^{2/3}~\hbox{ for }~p>2
~\hbox{ and }~ \#\cU_{2,2}(x)~\sim~c_2x/(\log x)^{2/3}
$$
as $x\to\infty$.
Using the same tools, we have
\begin{eqnarray*}
\#\cU_{2,1}(x)~\ll~x/(\log x)^{17/24},&&
\#\cU_{p,2}(x)~\ll~x/(\log x)^{5/6}~\hbox{ for }~p\ge3,\\
\#\(\bigcup_{k\ge3}\cU_{p,k}\)(x)&\ll&x/(\log x)^{5/6}~\hbox{ for all }~p.
\end{eqnarray*}
The result thus follows from Lemma~\ref{lem:structure2}.
\end{proof}

\begin{rem}
{\rm As mentioned in the introduction,
a more precise result, giving an asymptotic expansion
for $\#\cU_p(x)$ is presented by Moree~\cite[Theorem~5]{Mo}.}
\end{rem}

We need an estimate on the cardinality of
a somewhat more specialized set
which we  use in the sequel.
Suppose $m$ is an odd integer not divisible by
$p$.  Let
\begin{equation}
\label{eq:Qpmdef}
\cQ_{p,m}~=~\{r\in\cR_p:r\equiv 1\kern-5pt\pmod m\}.
\end{equation}

\begin{proposition}
\label{prop:Qpkdistribution}
Let $x$ be large. Assume that a prime $p$ and a positive odd
integer $m$ not divisible by $p$ satisfy the inequalities
$$
p \le   (\log x)^{2/3}
\quad \text{and}\quad m\le\frac{(\log x)^{1/6}}{\log_2x}.
$$
We have
$$
\#\cQ_{p,m}(x)~=~\frac1{3\varphi(m)}\pi(x) +O\(\frac{x\log_2x}{(\log
x)^{7/6}}\).
$$
\end{proposition}

\subsection{Chebotarev density theorem and its applications}

We let $\L$  be a  finite Galois extension of $\Q$ with
Galois group $G$ of degree $k= [\L:\Q]$ and discriminant
$\Delta$.  Let $\cC$ be a union of conjugacy classes of $G$.
We define
$$
\pi_\cC(x,\L/\Q) =\# \{ p \le x~:~p \text{ unramified in } \L/\Q, \
\sigma_p \in \cC\},
$$
where $\sigma_p$ is the Artin symbol of $p$ in the extension $\L/\Q$,
see~\cite{Gras}.

Combining a version of the
Chebotarev density theorem due to Lagarias and Odlyzko~\cite{LO}
together with a bound for a possible Siegel zero
due to Stark~\cite{Sta}, we obtain
the following result.

\begin{lemma}
\label{lem:cheb} There are absolute constants $A_1, A_2 >0 $ such that
if
\begin{equation}
\label{eq:necessarybound}
\log x~\ge~10 k (\log |\Delta|)^2
\end{equation}
then
$$ \left|\pi_\cC(x,\L/\Q) -  \frac{\#\cC}{\#G} \li(x)\right|~ \ll~
\frac{\#\cC}{\#G}
\li\(x^\beta\) +
\|\cC\| x \exp \(-A_1 \sqrt{\frac{\log x}{k}}\)
$$
with some $\beta$ satisfying the inequality
$$
\beta~<~1 - \frac{A_2}{\max\{|\Delta|^{1/k}, \log |\Delta|\}},
$$
where $\|\cC\|$ is the number of conjugacy classes in $\cC$.
\end{lemma}

We use Lemma~\ref{lem:cheb} in the proofs of
Propositions~\ref{prop:Rpkdistribution} and~\ref{prop:Qpkdistribution}.
It should be noted that in these applications we are studying
primes which split completely in certain normal extensions of $\Q$,
and so we might have gotten by with just Landau's prime ideal theorem.
However, to our knowledge the best explicit form of the prime ideal
theorem is that given in the more general Lemma~\ref{lem:cheb}.

In order to apply Lemma~\ref{lem:cheb} we need
an estimate for the discriminants of certain number fields $K\subset
L$, which we now present.
Let $\Delta(\L/\K)$ denote the relative discriminant of $\L$
over $\K$ and let $\Delta(\L)=\Delta(\L/\Q)$.

\begin{lemma}
\label{discriminantbound}
Let $n,d$ be positive integers with $d\mid n$ and let $a$ be
an integer with $|a| > 1$.  Let $h$ denote the largest integer for which
$a$ is an $h$-th power in $\Z$ and assume $\gcd(d,h)=1$.
For the field $\L=\Q(e^{2\pi i/n},a^{1/d})$, we have
$$
[\L:\Q]~=~d\varphi(n)~\hbox{ or }~d\varphi(n)/2,
\quad|\Delta(\L)|~\le~(d\varphi(n)|a|)^{[\L:\Q]}.
$$
Further, if $a=a_1a_2^2$ where $a_1$ is squarefree,
then  $[\L:\Q]=d\varphi(n)/2$ if and only if
$d$ is even and either $a_1\mid n$, $a_1\equiv1\pmod4$ or
$4a_1\mid n$, $a_1\not\equiv1\pmod4$.
\end{lemma}

\begin{proof}
The assertions about $[L:\Q]$ follows from~\cite[Lemma~2.2]{GP}
(for the case $d=n$, see also~\cite[Equations~(12) and~(13)]{Ho} 
and~\cite[Proposition~4.1]{Wag}).
Let $\K$ be the cyclotomic field $\Q(e^{2\pi i/n})$ and
write $[L:\Q]=d\varphi(n)/\vartheta$, where $\vartheta=1$ or~2. 
In particular if $\vartheta=2$, then $d$ is even and
$a^{1/2}\in K$.  Thus, the minimum polynomial for $a^{1/d}$
over $K$ is $x^{d/\vartheta}-a^{1/\vartheta}=f(x)$, say.
{}From elementary algebraic number theory we have
$$
\Delta(\L)~=~\Delta(\K)^{[\L:\K]}N_{K/\Q}(\Delta(\L/\K)).
$$
Now $\Delta(\L/\K)$ divides $N_{\L/\K}(f'(a^{1/d}))$
(see~\cite[Proposition~2.9]{Nar}) so that
$$
N_{K/\Q}(\Delta(\L/\K)) \mid
N_{K/\Q}(N_{\L/\K}(f'(a^{1/d})))
~=~N_{L/\Q}((d/\vartheta)a^{1/\vartheta-1/d}).
$$
Since each conjugate of $(d/\vartheta)a^{1/\vartheta-1/d}$
has absolute value $(d/\vartheta)|a|^{1/\vartheta-1/d}$,
we have
$$
|N_{K/\Q}(\Delta(\L/\K))|~\le~((d/\vartheta)|a|^{1/\vartheta-1/d})^{[L:\Q]}
~\le~(d|a|)^{[\L:\Q]}.
$$
It is well-known and easy to see from Hadamard's inequality for
determinants that $|\Delta(\K)|\le \varphi(n)^{\varphi(n)}$.
Thus $|\Delta(\K)|^{[\L:\K]}\le \varphi(n)^{[\L:\K]\varphi(n)}=
\varphi(n)^{[\L:\Q]}$. Assembling our estimates gives the lemma.
\end{proof}

For a prime $p$ and natural numbers $d,n$ with $d\mid n$, let
$$
\L_{p,n,d}~=~\Q(e^{2\pi i/n},p^{1/d})
$$
and let $\varpi_p(x;n,d)$ denote
the number of primes $r\le x$ with
$r\equiv1\pmod n$ and $d\mid(r-1)/\ell_p(r)$.
Thus, $\varpi_p(x;n,d)$ is the number of primes $r\le x$
which split completely in $\L_{p,n,d}$.  We may thus use
Lemmas~\ref{lem:cheb} and~\ref{discriminantbound} to estimate
$\varpi_p(x;n,d)$.

\begin{lemma}
\label{lem:varpiest}
For
$$
p \le   (\log x)^{2/3}
\quad \text{and}\quad n\le\frac{(\log x)^{1/6}}{\log_2x}
$$
and any number $A>0$,
we have
$$
\varpi_p(x;n,d)~=~
\frac1{[\L_{p,n,d}:\Q]}\li(x)+O_A\(\frac{x}{(\log x)^{A}}\).
$$
\end{lemma}
\begin{proof}
We apply Lemma~\ref{lem:cheb}
to the primes that split completely in $\L_{p,n,d}$.  Thus,
$\#\cC=1$ and $\#G= [\L_{p,n,d}:\Q]$.
Using Lemma~\ref{discriminantbound} and the assumptions on $p$ and $n$, we have
with $\Delta=\Delta(\L_{p,n,d})$,
$$
[\L_{p,n,d}:\Q] (\log|\Delta|)^2~\le~(d\varphi(n))^3 \(\log(dnp)\)^2
~\le~n^6 \(\log_2 x\)^2 ~=~o(\log x).
$$
Thus, for $x$ sufficiently large,
the condition~\eqref{eq:necessarybound} of Lemma~\ref{lem:cheb}
is satisfied.
Also
\begin{eqnarray*}
\max\{|\Delta|^{1/[\L_{p,n,d}:\Q]},\log|\Delta|\}
&\le&\max\{d\varphi(n)p,d\varphi(n)\log(dnp)\}\\
&\le&dn(\log x)^{2/3}~\le~n^2 (\log x)^{2/3}~\le~\frac{\log
x}{(\log_2x)^2}.
\end{eqnarray*}
Therefore,
$$
\beta~<~1-\frac{A_2(\log_2x)^2}{\log x},
$$
so that
$$
\li(x^{\beta})~\le~x^{\beta}~\le~\frac{x}{(\log x)^{A_2\log_2x}}.
$$
The second term in the inequality of Lemma~\ref{lem:cheb}
is smaller than this estimate under the above restriction on the
size of  $n$, so we have the lemma.
\end{proof}

\begin{rem}
{\rm One can reduce the limit for $p$
in Lemma~\ref{lem:varpiest} and get a much stronger
bound of the error term. However this does not affect our
main results.}
\end{rem}

\subsection{Proof of Propositions~\ref{prop:Rpkdistribution}
and~\ref{prop:Qpkdistribution}}
\label{sec:Prop1 proof}

We are now in a position to prove Proposition~\ref{prop:Rpkdistribution}.
For example, take the case of $\cR_{p,1}$ for $p>2$.  Let
\begin{eqnarray*}
N_{p,k}&=&
\varpi_p(x;2^k,2^{k-1})-\varpi_p(x;2^k,2^k)\\
& & -\(\varpi_p(x;2^{k+1},2^{k-1})
-\varpi_p(x;2^{k+1},2^k)\).
\end{eqnarray*}
Then $N_{p,k}$ is precisely the number of primes $r\le x$ with
$v_2(\ell_p(r))=1$, and $v_2(r-1)=k$.  Indeed, the first two terms
count those primes $r$ satisfying these conditions plus some additional
primes $r$ for which $v_2(r-1)>k$,
and the last two terms remove from the count these extra primes $r$.
Thus,
\begin{equation}
\label{eq:decomp}
\#\cR_{p,1}(x)~=~\sum_{k\ge1}N_{p,k}.
\end{equation}
By Lemma~\ref{discriminantbound} and also
Lemma~\ref{lem:varpiest} (used with $A=2$),
if $2^{k+1}\le (\log x)^{1/6}/\log_2x$, we have
\begin{equation}
\label{eq:Ncalc}
\begin{split}
N_{p,k}~=~&
\(\frac1{2^{2k-2}}-\frac1{2^{2k-1}}-\frac1{2^{2k-1}}
+\frac1{2^{2k}}\)\li(x)
+O\(\frac{x}{(\log x)^{2}}\)\\
~=~& \frac1{2^{2k}} \li(x)
+O\(\frac{x}{(\log x)^{2}}\).
\end{split}
\end{equation}
We apply~\eqref{eq:Ncalc} in~\eqref{eq:decomp} for those values of
$k$ with $2^{k+1}\le  (\log x)^{1/6}/\log_2x$, and for larger values of $k$
we use that by the Brun--Titchmarsh theorem,
see~\cite[Chapter~I.4, Theorem~9]{Ten},
$$
    N_{p,k}~ \le~ \pi(x;2^k,1)~\ll~\frac{\pi(x)}{2^k}
\quad \text{for }~2^k\le
x^{1/2},
$$
and also the elementary estimate
$$
    N_{p,k}~ \le~ \pi(x;2^k,1)~\le~\frac{x}{2^k},
$$
used when $2^k> x^{1/2}$.
We thus obtain
$$
\#\cR_{p,1}~=~\frac13\li(x)+O\(\frac{x\log_2x}{(\log x)^{7/6}}\)
~=~\frac13\pi(x)+O\(\frac{x\log_2x}{(\log x)^{7/6}}\)
$$
by the prime number theorem.

The remaining cases of Proposition~\ref{prop:Rpkdistribution} follow in
a similar manner, noting that when $p=2$ we can be in the situation
when $[\L_{p,n,d}:\Q]=d\varphi(n)/2$.

The same method can
be used to prove Proposition~\ref{prop:Qpkdistribution}.
Indeed, in the expression for $N_{p,k}$ put a factor $m$ in
the four middle arguments and then
use Lemma~\ref{lem:varpiest} if
$m 2^{k+1} \le (\log x)^{1/6}/\log_2x$,
the Brun--Titchmarsh theorem for $ (\log x)^{1/6}/\log_2x < m 2^{k+1}
\le x^{1/2}$,
and the trivial bound for $m 2^{k+1} >x^{1/2}$.
We suppress the details.

\subsection{Ranks of curves $\E_d$}
\label{sec:Rank}

We need the following inequality which allows us to
study the rank of $\E_d$ for an arbitrary $d \ge 1$.

\begin{lemma}
\label{lem:Rank} For positive integers $f,d$ with $f\mid d$, we have
$ R_q(d) \ge R_q(f)$.
\end{lemma}

\begin{proof}  It is clear that $\E_d$ contains the subgroup
of points $(x(t^g), y(t^g))$, where $g = d/f$.  This subgroup is isomorphic to
$\E_f$.
\end{proof}

\begin{rem}{\rm
It is clear from the definition of $I_q(d)$, that 
  if $f\mid d$ then  $I_q(d) \ge  I_q(f)$.}
\end{rem}

For $d$ a positive integer and $p$ a prime, let $d_p$ be the largest
divisor of $d$ whose every prime factor comes from $\cR_p$,
that is,
\begin{equation}
\label{eq:Def d_p}
d_p = \prod_{r\in \cR_p} r^{v_r(d)}.
\end{equation}
We are now able to combine Lemma~\ref{lem:Rank} with~\eqref{eq:R and I}
to get the following result.

\begin{proposition}
\label{prop:rank}
Let $\F_q$ be a finite field of characteristic $p$.
For every positive integer $d$ we have
$$
R_q(d)~\ge~\sum_{e\mid d_p}\frac{\varphi(e)}{\ell_q(e)}-4.
$$
\end{proposition}

Let $\lambda$ denote the Carmichael function; it is defined for each
integer  $d\ge 1$ as the largest order of an element in the
multiplicative group $(\Z/d\Z)^\times$. More explicitly, for any
prime power $l^\nu$, one has
$$
\lambda(l^\nu) = \left\{  \begin{array}{ll}
l^{\nu-1}(l-1),&\quad\hbox{if $l \ge 3$ or $\nu \le 2$},\\
2^{\nu-2},&\quad\hbox{if $l = 2$ and $\nu \ge 3$},
\end{array} \right.
$$
and for an arbitrary integer $d\ge 2$,
$$
\lambda(d)=\lcm\left[\lambda(l^{\nu})~:~l^\nu\mid d\right].
$$
Note that $\lambda(1)=1$.

If $d$ is coprime to $q$, then as is immediate from the definitions,
$$
\ell_q(d)~\le~ \lambda(d).
$$
We conclude from Proposition~\ref{prop:rank} that
for any finite field $\F_q$ of characteristic $p$ and any positive integer $d$,
we have
\begin{equation}
\label{eq:ranklowerbound}
R_q(d)~\ge~\frac{\varphi(d_p)}{\lambda(d_p)}-4.
\end{equation}

\section{Proof of Theorem~\ref{thm:Average}}

We begin with the upper bound~\eqref{eq:upper} since it is easier.
Note that
\begin{eqnarray*}
\sum_{d\in\cU_p(x)}R_q(d)
&\le&\sum_{d\in\cU_p(x)}I_q(d)
~=~\sum_{d\in\cU_p(x)}\sum_{e\mid d}\frac{\varphi(e)}{\ell_q(e)}\\
&\le&x\sum_{e\in\cU_p(x)}\frac{\varphi(e)}{e\ell_q(e)}
~\le~x\sum_{e\in\cU_p(x)}\frac1{\ell_q(e)}
~\le~x\sum_{n\le x}\frac1n\sum_{\substack{e\le x\\ \gcd(e,q)=1\\
\ell_q(e)=n}}1.
\end{eqnarray*}
In~\cite[Theorem~1]{Pom} it is shown that
$$
\sum_{\substack{m\le x\\ m~{\rm odd}\\ \ell_2(m)=n}}1
~\le~x^{1-(3+\log_3x)/(2\log_2x)}
$$
for all sufficiently large $x$, uniformly in $n$.
An examination of the proof shows
that for any integer $a$ and all sufficiently large $x$
depending only on $a$,
$$
\sum_{\substack{m\le x\\ \gcd(m,a)=1\\ \ell_a(m)=n}}1
~\le~x^{1-(3+\log_3x)/(2\log_2x)}
$$
for all $n$.
Using this estimate in the calculation above, we have
$$
\sum_{d\in\cU_p(x)}R_q(d)
~\le~x^{2-(3+\log_3x)/(2\log_2x)}\sum_{n\le x}\frac1n
~\le~x^{2-(2+\log_3x)/(2\log_2x)}
$$
for all sufficiently large $x$ depending on the choice of $q$.
Using Corollary~\ref{cor:Dpdistribution}
completes the proof of~\eqref{eq:upper}.

To prove the lower bound~\eqref{eq:alpha} in Theorem~\ref{thm:Average}
we loosely follow the construction from Erd\H os~\cite{Erd35} to construct
integers $v$ with many solutions to the equation $\varphi(n)=v$.
When $p$, the characteristic of $\F_q$, is odd,
let $u$ be an integer such that $u\equiv3\pmod4$ and the Legendre
symbol $(u/p)$ is $-1$;
and if $p=2$, let $u=5$.
Let $1/12>\delta>0$ be a small absolute constant to be chosen shortly,
let $z$ be large, and let
$$
\cI~=~[z^{1/2-2\delta},z^{1/2-\delta}],\quad
\cR~=~\{r\hbox{ prime }:~r\equiv u\kern-5pt\pmod{4p},\,
P(r-1)\in\cI\}.
$$
Note that any prime $r\equiv u\kern-3pt\pmod{4p}$ is in $\cR_{p}$, so
in particular,
we have $\cR\subset\cR_{p}$.  Let
$r,s,t$ denote prime variables.  We have
$$
\#\cR(z)~=~\sum_{s\in\cI}\,\sum_{\substack{r\le z\\ r\,\equiv\,
u\kern-5pt\pmod{4p}\\ r\,\equiv\,1\kern-5pt\pmod s}}1
~-~\sum_{s\in\cI}\,\sum_{s<t<z/s}\,
\sum_{\substack{r\le z\\ r\,\equiv\, u\kern-5pt\pmod{4p}\\
r\,\equiv\,1\kern-5pt\pmod{st}}}1
~=~S_1-S_2,
$$
say.  Indeed, any integer $n\le z$ is divisible
by either 0, 1, or 2 distinct primes that are greater than $z^{1/2-2\delta}$,
so $S_1$ counts 0, 1, or 2 correspondingly if $r-1$ has 0, 1, or 2
primes in $\cI$;
and $S_2$ makes the necessary correction in the case of 2 primes, or in the
case that $r-1$ is also divisible by a larger prime.

We now recall the Bombieri--Vinogradov theorem
which states
that for each  $A$ there is some number $B$ such that
\begin{equation}
\label{eq:B-V}
\sum_{m\le z^{1/2}/\log^Bz}\max_{\gcd(a,m)=1}
\left|\pi(z;m,a)-\frac1{\varphi(m)}\li(z)\right|~\ll~\frac{z}{\log^Az},
\end{equation}
see~\cite[Chapter~II.8, Theorem~11]{Ten}.

Using~\eqref{eq:B-V} and $p$ fixed, we have
by the Mertens formula
\begin{equation}
\label{eq:S1 asymp}
S_1~\sim~\frac{\log((1-2\delta)/(1-4\delta))}{\varphi(4p)}\pi(z)~\hbox{
as }~ z\to\infty.
\end{equation}
We reorganize $S_2$ by letting $(r-1)/st=a$, so that
$$
S_2~=~\sum_{a<z^{4\delta}}\,\sum_{s\in\cI}\,
\sum_{\substack{s<t<z/as\\ ast+1\,\equiv\, u\kern-5pt\pmod{4p}\\
     ast+1~{\rm prime}}}1.
$$
Note that since $z/as\ge z^{1/2-3\delta}$, we have by Brun's method
(see~\cite[Theorem~2.3]{HR}) that the double sum on $s$ and $t$ is
\begin{eqnarray*}
\sum_{s\in\cI}\,
\sum_{\substack{s<t<z/as\\ ast+1\,\equiv\, u\kern-5pt\pmod{4p}\\
     ast+1~{\rm prime}}}1
&\ll&\sum_{s\in\cI} \frac{z}{\varphi(4pas)\log^2(z/as)}\\
&\ll&\frac{\log((1-2\delta)/(1-4\delta))}{\varphi(4pa)}\frac{z}{\log^2z}.
\end{eqnarray*}
Thus,
$$
S_2~\ll~ \sum_{a<
z^{4\delta}}\frac{\log((1-2\delta)/(1-4\delta))}{\varphi(4pa)}
\frac{z}{\log^2z}~\ll
~\delta\frac{\log((1-2\delta)/(1-4\delta))}{\varphi(4p)}\pi(z),
$$
where we use the estimate
\begin{eqnarray*}
\sum_{a< Z}\frac{1}{\varphi(a)}& =&
\sum_{a< Z} \frac{1}{a} \sum_{d\mid a}\frac{\mu^2(d)}{\varphi(d)}
~\le~ \sum_{d< Z} \frac{1}{\varphi(d)} \sum_{b < Z/d}\frac{1}{db}\\
&\ll&\log Z\sum_d\frac1{\varphi(d)d} ~\ll~\log Z.
\end{eqnarray*}

Thus, there is an absolute choice for $\delta>0$ such that for all large
$z$ depending
on the choice of $p$, we have $S_2\le S_1/4$.  We now fix such a
value of $\delta$.
Note that the identity $\#\cR(z)=S_1-S_2$
and the asymptotic formula~\eqref{eq:S1 asymp}
applied to $z/2$ show that
$\#\cR(z/2)\le(1/2+o(1))S_1$.  We
conclude that for $z$ sufficiently large, depending on the choice of $p$, that
\begin{equation}
\label{eq:Rbound}
\#(\cR\cap[z/2,z])~\ge~\frac{\log((1-2\delta)/(1-4\delta))}{5\varphi(4p)}\pi(z).
\end{equation}

Let $x$ be large, and let
$$
       y ~=~  \frac{\log x}{\log_2 x}  \mand z ~=~ y^{2/(1-2\delta)} .
$$
Let $M_y$ denote the least common multiple of the integers in $[1,y]$ and
let
$$\cQ=\{r\in\cR\cap[z/2,z]~:~r-1\mid M_y\}.
$$
We note that for $r\in\cQ$, we have $P(r-1) \le y=z^{1/2-\delta}$.
The number of primes $r \le z$
such that $\ell^k| r-1$ for some prime power $\ell^k>y$ with $k\ge 2$
is bounded by
$$
\sum_{2\le k\le \log z/\log2}~~  \sum_{\ell\,:\,\ell^k \ge y}
\frac{z}{\ell^k} ~\ll
~z \sum_{2 \le k \le \log z/\log 2} \frac{1}{k y^{1-1/k}} ~\ll~ \frac{z \log
z}{y^{1/2}}.
$$
Combining this with~\eqref{eq:Rbound} we have
\begin{equation}
\label{eq:R'bound}
\#\cQ~\ge~\kappa \frac{z}{\log z}
\end{equation}
for $z$ sufficiently large depending on the choice of $p$,
where
$$
     \kappa~=~\frac{\log((1-2\delta)/(1-4\delta))}{6\varphi(4p)}.
$$

We now put
$$
m~ =~ \fl{\frac{\log x}{\log z}}
$$
and consider the set $\cS$ of all products of $m$ distinct primes from $\cQ$.
Clearly
\begin{equation}
\label{eq:range of d}
x~ \ge~ d~ \ge~ (z/2)^m~ =~ x^{1 + o(1)}
\end{equation}
for every $d \in \cS$.
     Recalling~\eqref{eq:R'bound}, we also have
\begin{eqnarray*}
\#\cS&=&\binom{\# \cQ}{m} ~\ge~ \(\frac{\# \cQ}{m}\)^m~
\ge~  \(\frac{ \kappa z}{\log x}\)^m~\ge~\frac1z \(\frac{\kappa z}{\log
x}\)^{\log x/\log z}\\
&=&x\exp\(-\frac{\log x}{\log z}(\log_2x+O(1))\)\\
&=&x\exp\big(-(1/2-\delta)\log x+O(\log
x\log_3x/\log_2x)\big)~=~x^{1/2+\delta+o(1)}.
\end{eqnarray*}
Note that for every $d \in \cS$ we have
$$
\ell_q(d) \mid \lambda(d) \mid  M_y.
$$
Thus, from the prime number theorem, we obtain that
$$
\ell_q(d)~\le~\exp((1+ o(1)) y) = x^{o(1)}.
$$
By the construction of $\cS$ and Lemma~\ref{lem:structure2} we have
$d\in\cUp$ so
that~\eqref{eq:R and I} can be applied to compute $R_q(d)$.
Therefore,~\eqref{eq:range of d} and a standard estimate for $\varphi(d)$ imply
that
$$
R_q(d)~\ge~I_q(d)-4~\ge~\frac{\varphi(d)}{\ell_q(d)}-4~=~\frac{d^{1+o(1)}}{x^{o(1)}}
~=~x^{1+o(1)}.
$$
Thus, using our estimate for $\#\cS$, we have
$$
\sum_{d\le x}R_q(d)~\ge~ x^{1+o(1)}\#\cS~\ge~ x^{3/2+\delta+o(1)}
$$
which concludes the proof.
\medskip

\begin{rem}
{\rm A key step in the proof is the use of the Bombieri--Vinogradov
theorem~\eqref{eq:B-V}.  We have
applied this result in the proof to moduli $4ps$ with
$s\in\cI$. The Elliott--Halberstam conjecture looks superficially the same, but
the range for $m$ is allowed to be much larger:  For every $\eps>0,A>0$,
$$
\sum_{m\le z^{1-\eps}}\max_{\gcd(a,m)=1}
\left|\pi(z;m,a)-\frac1{\varphi(m)}\li(z)\right|\ll\frac{z}{\log^Az}.
$$
Assuming this conjecture, the above proof gives Theorem~\ref{thm:Average}
for every value of $\alpha<1$.  The idea is similar to the proof of Theorem~3
in~\cite{AGP} and is also mentioned in~\cite{Granv}.  Let $k$ be an arbitrarily
large integer, let $\cI_k=[z^{1/k-1/k^2},z^{1/k}]$,
and let $\cR$ be the set of primes $r\equiv u\kern-3pt\pmod{4p}$ with
$r-1$ divisible by $k-1$ primes from $\cI_k$.  The primes $r\le z$
constructed in
this way have $P(r-1)\le z^\eta$, where $\eta=1-(k-1)^2/k^2$.
Further, by the Elliott--Halberstam
conjecture, there are at least $c_{k,p}\pi(z)$ such primes $r$,
where $c_{k,p}>0$ depends only on $k$ and $p$.  Let $y=\log x/\log_2x$ as
before and let $z=y^{1/\eta}$.  We do not have
to worry about taking only those values of $r$ that are $\ge z/2$, since
each $r$ is already guaranteed to be at least $z^{1-\eta}$, so that the values
of $d$ formed at the end of the proof are $\ge x^{1-\eta+o(1)}$.
Each of these values of $d$ has $l_q(d)\le x^{o(1)}$ as before, so that
$R_q(d)\ge x^{1-\eta+o(1)}$.  Moreover, as before,
there are $x^{1+o(1)}/\exp(\log x\log_2x/\log z)=x^{1-\eta+o(1)}$
values of $d$, so
that the average in Theorem~\ref{thm:Average} is at least $x^{1-2\eta+o(1)}$.
Since $k$ is arbitrary, this then proves that the average is $x^{1+o(1)}$.}
\end{rem}

\section{Proof of Theorem~\ref{thm:Normal}}

Our proof closely follows the proof of Theorem 2 in~\cite{EPS}.
This result gives the normal order of $\lambda(n)$,
showing that for almost all $n$ (that is, on a set of asymptotic
density 1), we have
$\lambda(n)=n/(\log n)^{(1+o(1))\log_3n}$.  Since for all $n$ we have
$n\ge\varphi(n)\gg n/\log_2n$, it follows
that for almost all $n$ we have
$$
\frac{\varphi(n)}{\lambda(n)}~=~(\log n)^{(1+o(1))\log_3n}
$$
as $n\to\infty$.

We first note the elementary fact that
\begin{equation}
\label{eq:div}
m\mid
n~\implies~\frac{\varphi(m)}{\lambda(m)}~\bigg|~\frac{\varphi(n)}{\lambda(n)}.
\end{equation}
Indeed, by the Chinese remainder theorem, there is an integer $a$ such
that for each prime power $l^\nu\mid n$ we have $\ell_a(l^\nu)=\lambda(l^\nu)$.
Then $\ell_a(n)=\lambda(n)$ and $\ell_a(m)=\lambda(m)$.  The canonical
epimorphism from $(\Z/n\Z)^\times$ to $(\Z/m\Z)^\times$ induces an
epimorphism from $(\Z/n\Z)^\times/\langle a\rangle$ to
$(\Z/m\Z)^\times/\langle a\rangle$, so that~\eqref{eq:div} follows.

Let $x$ be large and let $y=\log_2x$.
In view of~\eqref{eq:ranklowerbound}, it suffices to show that
\begin{equation}
\label{eq:logbound}
\log\varphi(d_p)-\log\lambda(d_p)~=~\frac13 y\log y+O_p(y\log_2y)
\end{equation}
for all $d\le x$ with at most $o_p(x)$ exceptions, where $d_p$ is given
by~\eqref{eq:Def d_p}.
(In fact~\eqref{eq:logbound} is somewhat stronger than required in that
we really only need a lower bound for the left side.  Nevertheless it
is interesting to know the true order of $\varphi(d_p)/\lambda(d_p)$
for almost all integers $d$.)
For all $d$ we have
$$
\log\varphi(d_p)~=~\sum_l v_l(\varphi(d_p))\log l,\quad
\log\lambda(d_p)~=~\sum_l v_l(\lambda(d_p))\log l,
$$
where the sums are over all primes $l$.
It follows from~(6) and~(19) in~\cite{EPS} that
$$
\sum_{l\le y\log y}v_l(\lambda(d_p))\log l
~\le~\sum_{l\le y\log y}v_l(\lambda(d))\log l
~=~ y\log_2 y+O(y)
$$
for all but $o(x)$ values of $d\le x$.
Using~\eqref{eq:div}, we have for each prime $l$,
$$
v_l(\varphi(d_p))-v_l(\lambda(d_p))
~\le~v_l(\varphi(d))-v_l(\lambda(d)).
$$
Also, from (20), (21), and (22) in~\cite{EPS} we have
$$
\sum_{l> y\log y}\(v_l(\varphi(d))-v_l(\lambda(d))\)\log l
~\le~\frac{y\log_2y}{\log y}+(\log y)^2
$$
for all but $o(x)$ values of $d\le x$.
It thus follows that
$$
\sum_{l> y\log y}\(v_l(\varphi(d_p))-v_l(\lambda(d_p))\)\log l
~\le~\frac{y\log_2y}{\log y}+(\log y)^2
$$
for all but $o(x)$ values of $d\le x$.
Thus, to prove that~\eqref{eq:logbound} holds for all but $o_p(x)$
values of $d\le x$, it suffices to show that
\begin{equation}
\label{eq:logphibound}
\sum_{l\le y\log y}v_l(\varphi(d_p))\log l~=~\frac13 y\log y+O_p(y\log_2y)
\end{equation}
holds for all but $o_p(x)$ values of $d\le x$.

  We  prove~\eqref{eq:logphibound} using the
Tur\'an--Kubilius inequality, arguing along the same
lines as in~\cite{EPS}.
We recall, that for real-valued additive
functions $g(n)$ the Tur\'an--Kubilius inequality asserts that if
$$
E(g,x)~=~\sum_{r^\nu\le x}\frac{g(r^\nu)}{r^\nu}\(1-\frac1r\) \mand
V(g,x)~=~\sum_{r^\nu\le x}\frac{g(r^\nu)^2}{r^\nu},
$$
then
\begin{equation}
\label{eq:TK}
\sum_{n\le x}\(g(n)-E(g,x)\)^2~\le~10xV(g,x),
\end{equation}
see~\cite[Chapter~III.3, Theorem~1]{Ten}.  Let
$$
h(n)~=~\sum_{l\le y\log y}v_l(\varphi(n))\log l,\quad
h_p(n)~=~h(n_p)~=~\sum_{l\le y\log y}v_l(\varphi(n_p))\log l,
$$
so that $h$ and $h_p$ are both additive functions.
It is shown in~\cite[pp.~366--367]{EPS}  that
$$
V(h,x)~\ll~y(\log y)^2.
$$
Since $V(h_p,x)\le V(h,x)$, we have $V(h_p,x)\ll y(\log y)^2$.

For the determination of $E(h_p,x)$ we use
Proposition~\ref{prop:Qpkdistribution}.
Since $h_p(r^\nu)\le \log(r^\nu)$, we have
$$
     E(h_p,x)~=~\sum_{r^\nu\le x}\frac{h_p(r^\nu)}{r^\nu}\(1-\frac1r\)
~=~\sum_{r\le x}\frac{h_p(r)}{r}+O(1).
$$
Now
$$
\sum_{r\le x}\frac{h_p(r)}{r}
~=~\sum_{l\le y\log y}\sum_{\substack{r\le x\\
r\in\cR_p}}\frac{v_l(r-1)\log l}{r}
~=~\sum_{l\le y\log y}\log l\sum_{i\ge1}
\sum_{\substack{r\le x\\ r\in\cR_p\\ v_l(r-1)=i}}\frac{i}{r}.
$$
The inner sum is $O\(iy/l^i\)$, so the contribution for values of $i>1$
is $O(y)$.  We conclude that
\begin{equation}
\label{eq:Eest}
     E(h_p,x)~=~\sum_{l\le y\log y}\log l\sum_{\substack{r\le x\\
r\in\cR_p\\ r\equiv
1\kern-5pt\pmod l}}
\frac1r +O(y).
\end{equation}
Recall the notation $\cQ_{p,m}$ from~\eqref{eq:Qpmdef}.
We use partial summation on the inner sum in~\eqref{eq:Eest} getting
$$
\sum_{r\in\cQ_{p,l}(x)} \frac1r
~=~\frac{\#\cQ_{p,l}(x)}{x}+\int_2^x\frac{\#\cQ_{p,l}(z)}{z^2}\,dz.
$$
We use the estimate $\#\cQ_{p,l}(z)\le\pi(z;l,1)\ll \pi(z)/l$
for $z\le\exp(l^7)$, and we use Proposition~\ref{prop:Qpkdistribution} for
larger values of $z$, getting that
$$
\sum_{r\in\cQ_{p,l}(x)} \frac1r
~=~\frac{y}{3(l-1)}+O\(\frac{\log l}{l}\).
$$
Putting this into~\eqref{eq:Eest} we get that
$$
     E(h_p,x)~=~\sum_{l\le y\log y}\frac{y\log l}{3(l-1)}+O(y)
~=~\frac13 y\log(y\log y)+O(y).
$$

We now use this estimate for $ E(h_p,x)$ and our earlier estimate for
$V$ in the
Tur\'an-Kubilius inequality~\eqref{eq:TK} applied to the function $h_p$.
We get that the number of $d\le x$ with
$$
\left|h_p(d)-\frac13y\log y\right|~>~y\log_2y
$$
is $O\( xy(\log y)^2/(y\log_2y)^2\) = o(x)$.  This concludes the proof
of~\eqref{eq:logphibound} and so proves the theorem.

\end{document}